\documentclass{amsart}

\usepackage{amsmath,amstext,amsthm,amssymb,amsfonts}
\usepackage[square,numbers]{natbib}     
\usepackage{fullpage}
\usepackage[all]{xy}

\usepackage[bbgreekl]{mathbbol}

\theoremstyle{plain}

\newtheorem{thm}{Theorem}
\newtheorem{lemm}{Lemma}

\DeclareMathOperator{\N}{{\mathbb{N}}}

\DeclareMathOperator{\CCat}{\mathcal{C}}
\DeclareMathOperator{\MCat}{\mathcal{M}}

\DeclareMathOperator{\EOp}{\mathsf{E}}
\DeclareMathOperator{\POp}{\mathsf{P}}
\DeclareMathOperator{\QOp}{\mathsf{Q}}
\DeclareMathOperator{\ROp}{\mathsf{R}}
\DeclareMathOperator{\SOp}{\mathsf{S}}
\DeclareMathOperator{\MOp}{\mathsf{M}}
\DeclareMathOperator{\NOp}{\mathsf{N}}
\DeclareMathOperator{\FOp}{\mathcal{F}}

\DeclareMathOperator{\Seq}{\mathcal{S}\mathit{eq}}
\DeclareMathOperator{\Op}{\mathcal{O}\mathit{p}}
\DeclareMathOperator{\Map}{\mathtt{Map}}
\DeclareMathOperator{\Mor}{\mathtt{Mor}}
\DeclareMathOperator{\End}{\mathtt{End}}
\DeclareMathOperator{\Ob}{\mathtt{Ob}}
\DeclareMathOperator{\pt}{\mathit{pt}}
\DeclareMathOperator{\Diag}{\mathrm{Diag}}
\DeclareMathOperator{\id}{\mathit{id}}
\DeclareMathOperator{\DGB}{\mathtt{B}}
\DeclareMathOperator{\DGL}{\mathbb{L}}

\begin{document}

\title{The homotopy theory of operad subcategories}

\author{Benoit Fresse}
\address{Universit\'e de Lille\\
Laboratoire Painlev\'e\\
Cit\' e Scientifique - B\^ atiment M2\\
F-59655 Villeneuve d'Ascq Cedex, France}
\email{Benoit.Fresse@math.univ-lille1.fr}

\author{Victor Turchin}
\address{Department of Mathematics\\
Kansas State University\\
138 Cardwell Hall\\
Manhattan, KS 66506, USA}
\email{turchin@ksu.edu}

\author{Thomas Willwacher}
\address{Department of Mathematics \\ ETH Zurich \\
R\"amistrasse 101 \\
8092 Zurich, Switzerland}
\email{thomas.willwacher@math.ethz.ch}

\thanks{B.F. acknowledges support by Labex ANR-11-LABX-0007-01 \lq\lq{}CEMPI\rq\rq{}.
V.T. is partially supported by the Simons Foundation travel grant, award ID: 519474.
T.W. has been partially supported by the Swiss National Science foundation, grant 200021-150012, the NCCR SwissMAP funded by the Swiss National Science foundation, and the ERC starting grant GRAPHCPX. The authors are grateful to the referee for her/his thorough reading of the paper}

\date{March 17, 2017 (revised on December 11, 2017 and February 5, 2018)}

\begin{abstract}
We study the subcategory of topological operads $\POp$ such that $\POp(0) = *$ (the category of unitary operads in our terminology).
We use that this category inherits a model structure, like the category of all operads in topological spaces,
and that the embedding functor of this subcategory of unitary operads into the category of all operads
admits a left Quillen adjoint.
We prove that the derived functor of this left Quillen adjoint functor induces a left inverse of the derived functor of our category embedding at the homotopy category level.
We deduce from this result that the derived mapping spaces associated to our model category of unitary operads
are homotopy equivalent to the standard derived operad mapping spaces,
which we form in the model category of all operads
in topological spaces.
We prove that analogous statements hold for the subcategory of $k$-truncated unitary operads within the model category of all $k$-truncated operads,
for any fixed arity bound $k\geq 1$, where a $k$-truncated operad denotes an operad
that is defined up to arity~$k$.
\end{abstract}

\maketitle

\section*{Introduction}

Let $\Op$ be the category of all operads in topological spaces. Throughout this paper, we use the terminology `\emph{unitary operad}',
borrowed from the book~\cite{FresseBook},
to refer to the category of operads $\POp$ satisfying $\POp(0) = *$,
and we adopt the notation $\Op_*\subset\Op$
for this subcategory of operads.
Recall that this category $\Op_*$ is isomorphic to the category of $\Lambda$-operads $\Lambda\Op_{\varnothing}$,
whose objects are operads satisfying $\POp(0) = \varnothing$, but which we equip with restriction operators $u^*: \POp(n)\rightarrow\POp(m)$,
associated to the injective maps $u: \{1<\dots<m\}\rightarrow\{1<\dots<n\}$,
and which model composition operations with an arity zero operation
at the inputs $j\not\in\{u(1),\dots,u(m)\}$ (see~\cite[\S I.3]{FresseBook}).

Let $\iota: \Op_*\hookrightarrow\Op$ be the obvious category embedding. This functor has a left adjoint $\tau: \Op\rightarrow\Op_*$,
which we call the \emph{unitarization}
in what follows.
In short, for an operad $\POp\in\Op$, the operad $\tau\POp$ is defined by collapsing $\POp(0)$ to a one-point set $\tau\POp(0) = *$,
and by taking the quotient of the spaces $\POp(r)$ under appropriate equivalence relations
in order to make all composites $p\circ_i e$ of a given element $p\in\POp(r+1)$ with an arity zero element $e\in\POp(0)$
equivalent to the same point $p\circ_i *$
in the operad $\tau\POp(r)$.

We use that both categories $\Op_*$ and $\Op$ can be equipped with a model structure in order to do homotopy theory.
We can use the model structures defined in~\cite{BergerMoerdijk} for the category $\Op$.
We then assume that the weak-equivalences and the fibrations of operads are created arity-wise in the base category of spaces.
We refer to this model structure on our category of operads $\Op$ as the projective model structure.
We can use the same construction to get a model structure on the category of unitary operads $\Op_*$ (see also~\cite{BergerMoerdijk},
where the terminology `reduced operad' is used for our category of `unitary operads').
We can also consider the Reedy model structure of~\cite[\S II.8.4]{FresseBook}, which is defined for the category of $\Lambda$-operads $\Lambda\Op_{\varnothing}$,
but which we can merely transport to the category of unitary operads $\Op_*$
by using the isomorphism of categories $\Lambda\Op_{\varnothing}\simeq\Op_*$. Both choices are equivalent for our purpose, and we can equip $\Op_*$ with the projective model structure
or with the Reedy model structure without any change in our arguments. Recall simply that the Reedy model structure
has less fibrations than the projective structure, but the identity functor gives a Quillen equivalence
between these model structures on $\Op_*$ (see again~\cite[\S II.8.4]{FresseBook}).

Both categories $\Op_*$ and $\Op$ are used in homotopy computations. To be specific, the authors use the model category $\Op_*$
to compute mapping spaces of $E_n$-operads in \cite{FresseTurchinWillwacher},
whereas the model category $\Op$ is used in the Goodwillie-Weiss calculus,
in the expression of the relationship between the mapping spaces
of $E_n$-operads
and the embedding spaces of Euclidean spaces with compact support (see for instance~\cite{BoavidaWeiss,DwyerHess,Turchin}).
The approach of \cite{FresseTurchinWillwacher} is to use an operadic enhancement of the Sullivan model,
which can be handled in the category $\Op_*$ (see~\cite[\S II.12]{FresseBook}),
in order to perform rational homotopy computations
in the category of operads.

The main purpose of this paper is to establish the following comparison statement, where we consider the derived mapping spaces $\Map^h_{\CCat}(-,-)$
associated to our model categories of operads $\CCat = \Op,\Op_*$:

\begin{thm}\label{thm:MainResult}
The functor $\iota: \Op_*\hookrightarrow\Op$ induces a weak-equivalence on derived mapping spaces:
\begin{equation*}
\Map_{\Op_*}^h(\POp,\QOp)\sim\Map_{\Op}^h(\iota\POp,\iota\QOp),
\end{equation*}
for all operads $\POp,\QOp\in\Op_*$.
\end{thm}

The derived mapping spaces of this theorem can be defined as usual, by taking the ordinary mapping spaces associated to a cofibrant resolution of our source object
and a fibrant resolution of our target object in our model categories.
For instance, we have $\Map_{\Op_*}^h(\POp,\QOp) := \Map_{\Op_*}(\ROp,\SOp)$, where $\ROp\xrightarrow{\sim}\POp$
is a cofibrant resolution of the operad $\POp$ in $\Op_*$,
whereas $\QOp\xrightarrow{\sim}\SOp$ is a fibrant resolution of the operad $\QOp$. We proceed similarly in the case
of the mapping space $\Map_{\Op}^h(\iota\POp,\iota\QOp)$.
But we do not really need to make these mapping space constructions more explicit, because we deduce our statement
from another approach which involves the left adjoint $\tau: \Op\rightarrow\Op_*$
of our category embedding $\iota: \Op_*\hookrightarrow\Op$.

In brief, we readily see that these functors define a Quillen adjunction $\tau: \Op\rightleftarrows\Op_* :\iota$ whatever choice we make for the model
structure on $\Op_*$ (the projective model structure or the Reedy model structure).
This Quillen adjunction relation implies that we have a weak-equivalence
at the derived mapping space level
\begin{equation*}
\Map_{\Op}^h(\iota\POp,\iota\QOp)\sim\Map_{\Op_*}^h(\DGL\tau(\iota\POp),\QOp),
\end{equation*}
where $\DGL\tau$ denotes the left derived functor of the left Quillen adjoint $\tau: \Op\rightarrow\Op_*$.
Then we can reduce the proof of Theorem~\ref{thm:MainResult} to the verification
that we have the relation $\DGL\tau(\iota\POp)\xrightarrow{\sim}\POp$
when we consider the augmentation morphism
of the derived adjunction relation
associated to our functors.
We have by definition $\DGL\tau(\iota\POp) := \tau\ROp$, where $\ROp\xrightarrow{\sim}\iota\POp$ is any cofibrant resolution of the object $\iota\POp$
in the model category $\Op$.
We are therefore left to verifying that we get a weak-equivalence $\tau\ROp\xrightarrow{\sim}\POp$,
for a good choice of the resolution $\ROp\xrightarrow{\sim}\iota\POp$,
when we pass to the category of unitary operads $\Op_*$.
In what follows, we generally omit to mark the functor $\iota: \POp\mapsto\iota\POp$ in our formulas.

We can assume that $\POp$ is the geometric realization of an operad in simplicial sets, which we abusively denote by the same letter $\POp$,
because the model category of operads in simplicial sets is Quillen equivalent to the model category of topological operads
(see for instance \cite[\S II.1.4]{FresseBook}).
We consider an arity-wise cartesian product $\EOp\times\POp$, where $\EOp$ is an $\EOp_{\infty}$-operad in simplicial sets,
to get an operad in simplicial sets equipped with a free action of the symmetric group
and such that $\EOp\times\POp\xrightarrow{\sim}\POp$. (We make our choice of this operad $\EOp$ explicit later on.)
We take $\ROp = W(\EOp\times\POp)$, the Boardman-Vogt construction on this operad $\EOp\times\POp$, as a cofibrant resolution of the object $\POp$
in the model category of operads $\Op$.
We actually check that we have the following statement:

\begin{thm}\label{thm:Goal}
We have $\tau W(\EOp\times\POp)\xrightarrow{\sim}\EOp\times\POp\xrightarrow{\sim}\POp$, for any unitary operad $\POp\in\Op_*$.
\end{thm}

Thus, if we recap our arguments, then we can reduce the proof of Theorem~\ref{thm:MainResult} to the verification of this claim,
and we devote the next section to this objective.
In short, we mainly prove that the operad morphism $\tau W(\EOp\times\POp)\rightarrow\EOp\times\POp$
defines a homotopy equivalence of simplicial sets arity-wise.
We have a classical contracting homotopy on the $W$-construction (see for instance~\cite[\S III.1]{BoardmanVogt}).
We can use this homotopy to check that the object $\ROp = W(\EOp\times\POp)$ is equivalent to $\EOp\times\POp\sim\POp$
in the homotopy category of operads,
but this homotopy does not pass to the quotient object $\tau W(-)$ (see~\S\ref{subsec:homotopy remark}).
We introduce another contracting homotopy in order to work out this problem and to prove our theorem.
The consideration of the cartesian product $\EOp\times\POp$, where $\EOp$ is an $\EOp_{\infty}$-operad in simplicial sets,
enables us to ensure, at first, that the object $\ROp = W(\EOp\times\POp)$
is cofibrant in the projective model category of operads.
But, actually, we use a particular choice of the operad $\EOp$ in order to get a well-defined contracting homotopy of simplicial sets
on the operad $\tau W(\EOp\times\POp)$.
We give more explanations on this technical point in the course or our verifications.

To complete our results, we establish an analogue of our main theorems for the categories of $k$-truncated operads, considered in the paper \cite{FresseTurchinWillwacher}.
The category of $k$-truncated operads, where we fix $k\geq 1$, explicitly consists of the operads $\POp$ that are defined up to arity $k$.
We check that our argument lines can be adapted to cover this case.
We devote a second section to this survey. We use mapping spaces of $k$-truncated operads
for the study of the Goodwillie-Weiss tower of embedding spaces
in \cite{FresseTurchinWillwacher}, and the $k$-truncated refinement of our comparison result
is involved in such applications.

\section{Proof of the main statements}
The goal of this section is to prove Theorem~\ref{thm:Goal}, and as a follow-up Theorem~\ref{thm:MainResult},
as we just explained in the paper introduction.
We make explicit our choice of the $E_{\infty}$-operad $\EOp$ first.
We review the definition of the $W$-construction afterwards and we eventually give this proof of Theorem~\ref{thm:Goal}.

\subsection{The extended Barratt-Eccles operad}\label{subsec:BarrattEccles}
We use the notation $\Sigma_r$ for the symmetric group in $r$ letters all along this paper, for any $r\in\N$.

The operad $\EOp$ which we consider in our construction is a simple extension of the classical Barratt-Eccles operad~\cite{BarrattEccles}.
We define this operad as the classifying space $\EOp = \DGB(\MCat)$
of a certain operad in the category of categories $\MCat$.
We first take:
\begin{equation}
\Ob\MCat = \FOp(\mu(x_1,x_2),\mu(x_2,x_1),e)/\langle\mu\circ_1\mu\equiv\mu\circ_2\mu,\mu(e,e)\equiv e\rangle,
\end{equation}
the operad in sets generated by a non-symmetric operation $\mu = \mu(x_1,x_2)$ in arity $2$,
and an operation $e$ in arity $0$,
together with the associativity relation $\mu(\mu(x_1,x_2),x_3)\equiv\mu(x_1,\mu(x_2,x_3))$
and the idempotence relation $\mu(e,e)\equiv e$
as generating relations.
We then set:
\begin{equation}
\Mor_{\MCat(r)}(p,q) = \pt,
\end{equation}
for any pair of elements $p,q\in\Ob\MCat(r)$. We define the composition operations of this operad $\circ_i: \MCat(k)\times\MCat(l)\rightarrow\MCat(k+l-1)$
by the natural composition operations of the operad $\Ob\MCat$
at the object set level, and by the obvious identity maps
at the morphism set level. (We give brief explanations on an interpretation of this operad in a remark at the end of the paper.)

If we use classical algebraic notation for the product operation $x_1 x_2 = \mu(x_1,x_2)$,
then we can identify the elements of $\Ob\MCat(r)$ with monomials of the form:
\begin{equation}
p(x_1,\dots,x_r) = e^{\epsilon_0} x_{\sigma(1)} e^{\epsilon_1} \cdots e^{\epsilon_{r-1}} x_{\sigma(r)} e^{\epsilon_r},
\end{equation}
where $\epsilon_0,\dots,\epsilon_r\in\{0,1\}$, and $\sigma\in\Sigma_r$.
The operadic composition operations are given, at the object-set level, by the standard substitution operation
of monomials together with the Boolean multiplication rules:
\begin{equation}
e^{\alpha} e^{\beta} = \begin{cases} e^0, & \text{if $\alpha = \beta = 0$}, \\
e^1, & \text{otherwise}.
\end{cases}
\end{equation}

The constant maps $\MCat(r)\rightarrow\pt$ trivially define equivalences of categories in all arities $r\in\N$,
and as a consequence, these maps induce a weak-equivalence
of operads in simplicial sets when we pass to classifying spaces:
\begin{equation}
\EOp = \DGB(\MCat)\xrightarrow{\sim}\pt.
\end{equation}

Let $\overline{\EOp}_n = \overline{\DGB(\MCat)}_n$ denote the collection of sets, where $n\in\N$ is a fixed simplicial dimension,
which we form by dropping (the degeneracies of) the vertex $1\in\Ob\MCat(1)$ (the operadic unit)
from $\overline{\DGB(\MCat)}_n$.
The sets $\overline{\DGB(\MCat)}(1)_n$, $n\in\N$, are not preserved by the face operators
of the classifying space $\DGB(\MCat)(1)$.
For instance, if we take the morphism $x_1 e\rightarrow x_1$, which represents a one simplex in $\DGB(\MCat)(1)$,
then we have $d_0(x_1 e\rightarrow x_1) = x_1 = 1$.
But, on the other hand, we have the following lemma:

\begin{lemm}\label{lemm:nonunital operad structure}
The collections $\overline{\EOp}_n = \overline{\DGB(\MCat)}_n$, $n\in\N$, are preserved by the operadic composition operations
of the operads in sets $\DGB(\MCat)_n$, and hence, form operads without unit.
\end{lemm}

In fact, this statement gives the main property of the operad $\EOp = \DGB(\MCat)$ which we use in our subsequent constructions.

\begin{proof}
The algebraic description of the object-set operad $\Ob\MCat$ shows that
this property holds for the collection
of vertex sets $\overline{\Ob\MCat} = \overline{\DGB(\MCat)}_0$.
The conclusion of the lemma follows.
\end{proof}

\subsection{The $W$-construction and its unitarization}\label{subsec:Wconstruction}
We assume that $\POp$ is an operad in topological spaces for the moment. Briefly recall that the spaces $W(\POp)(r)$ underlying the $W$-construction $W(\POp)$
consist of collections $[T;p_x,x\in VT;l_e,e\in\mathring{ET}]$,
where $T$ is an $r$-tree (a tree with $r$ ingoing edges numbered from $1$ to $r$),
while $p_x\in\POp(r_x)$ is an operation associated to each vertex $x\in VT$ (whose number of ingoing edges is denoted by $r_x$),
and $l_e\in [0,1]$ is a length associated to each inner edge $e\in\mathring{ET}$ (see~\cite{BergerMoerdijkW,BoardmanVogt}).
To represent such a collection, we generally use a decoration of the tree $T$,
with the elements $p_x\in\POp(r_x)$ on the vertices $x\in VT$,
and the parameters $l_e\in [0,1]$
on the inner edges $e\in\mathring{ET}$.
If we have $l_{\alpha} = 0$ for some internal edge $\alpha\in\mathring{ET}$
with $v$ as target vertex and $u$ as source vertex,
then we have the relation:
\begin{equation}\label{eq:edge contraction}
\vcenter{\xymatrix@R=1.5em{ & \ar@{-}[dr] & \cdots & \ar@{-}[dl] \\
\ar@{-}[dr] & \cdots & p_v\ar@{-}[dl]|{l_{\alpha}=0} & \\
& p_u\ar@{-}[d] && \\
& \cdots && }}\equiv\vcenter{\xymatrix@R=1.5em{ \ar@{-}[dr] & \cdots & \ar@{-}[dl] & \cdots & \ar@{-}[dlll] \\
& p_u\circ_{\alpha} p_v\ar@{-}[d] &&& \\
& \cdots &&& }},
\end{equation}
where we contract the edge $\alpha$ in $T$, and we perform the composition operation $p_u = p_v\circ_{\alpha} p_w$
to get an element of $W(\POp)$ shaped on the tree $T/\alpha$.
In $W(\POp)$, we also implement the relation:
\begin{equation}\label{eq:unit contraction}
\vcenter{\xymatrix@R=1.5em{ \cdots\ar@{-}[d]_{l_{\beta}} \\
1\ar@{-}[d]_{l_{\alpha}} \\ \cdots }}\equiv\vcenter{\xymatrix{ \cdots\ar@{-}[d]^{\max(l_{\alpha},l_{\beta})} \\ \cdots }}
\end{equation}
when we have a vertex labeled by the unit element of the operad $1\in\POp(1)$.
The operadic composite $\circ_i$ of elements shaped on decorated trees $S$ and $T$ in $W(\POp)$
is obtained by plugging the outgoing edge of the tree $T$ in the $i$th ingoing edge
of $S$,
and by assigning the length $l=1$ to this new inner edge of the composite tree $S\circ_i T$.
Recall that we have a weak-equivalence $W(\POp)\xrightarrow{\sim}\POp$, defined by forgetting about the length of the edges
and by performing the composition operations shaped on our trees
in the operad $\POp$.

In $\tau W(\POp)$, we implement the extra reduction relation
\begin{equation}\label{eq:reduction relation}
\vcenter{\xymatrix@R=1.5em{ *+<2pt>{*}\ar@{-}[dr] & \cdots\ar@{}[d]|{\textstyle S} & *+<2pt>{*}\ar@{-}[dl] \\
& \ar@{-}[d]^{l_0=1} & \\
& \cdots & }}\equiv\vcenter{\xymatrix@R=1.5em{ *+<2pt>{*}\ar@{-}[d]^{l=1} \\
\cdots }},
\end{equation}
when we have a whole subtree $S$ with an outgoing edge of length $l_0 = 1$ in which all chains of edges abut to an element
of arity zero $*\in\POp(0)$.

\subsection{The cofibrant structure of the $W$-construction and the reduction to a non-unital $W$-construction}\label{subsec:nonunitalWconstruction}
We already mentioned that the model category of operads in simplicial sets
is Quillen equivalent to the model category of operads
in topological spaces.
Therefore, we now assume (without loss of generality in our statement) that $\POp$ is an operad in simplicial sets,
of which we can take the geometric realization $|\POp|$ to pass to the category of operads
in topological spaces.
Then we consider the operad $\EOp\times\POp$ such that $(\EOp\times\POp)(r) = \EOp(r)\times\POp(r)$, for any $r\in\N$.
We can still take the geometric realization of this operad to get an operad in topological spaces $|\EOp\times\POp|$.
We then have $\EOp\sim\pt\Rightarrow\EOp\times\POp\sim\POp\Rightarrow|\EOp\times\POp|\sim|\POp|$.
We moreover get that the $W$-construction of this operad $W(|\EOp\times\POp|)$
is cofibrant as an operad in topological spaces, which is not the case of the operad $W(|\POp|)$
in general (when the symmetric groups do not operate
freely on the components of the operad $\POp$).
We just review the proof of a counterpart of this claim in the category of operads in simplicial sets in the next paragraph.
We refer \cite{BergerMoerdijkW} for a more detailed study of the definition and properties of the $W$-construction
in the general setting of model categories.

We can actually stay in the category of operads in simplicial sets for our study, because we have an identity $W(|\EOp\times\POp|) = |W(\EOp\times\POp)|$,
where we consider a simplicial version of the $W$-construction $W(\EOp\times\POp)$
before passing to the geometric realization.
In order to adapt the definition of the $W$-construction to the simplicial setting, and hence, in order to define
this object $W(\EOp\times\POp)$, we just replace the interval $[0,1]$ by the $1$-simplex $\Delta^1$
in the definitions
of the previous paragraph~(\S\ref{subsec:Wconstruction}).
Note simply that the unit reduction relation (Eq. \ref{eq:unit contraction}) involves the geometric realization
of a simplicial map $m: \Delta^1\times\Delta^1\rightarrow\Delta^1$,
so that we can still give a sense to this relation
within the category of simplicial sets.

The operad $W(\EOp\times\POp)$, which we obtain by taking this simplicial $W$-construction, is cofibrant as an operad in simplicial sets (just like
the topological $W$-construction $W(|\EOp\times\POp|)$ of the operad $|\EOp\times\POp|$
is cofibrant as an operad in topological spaces).
This assertion can be deduced from the observation that the operad $W(\EOp\times\POp)$
admits a free structure in each simplicial dimension.
To be more precise, we have an identity $W(\EOp\times\POp)_n = \FOp(\mathring{W}(\EOp\times\POp)_n)$ in every simplicial dimension $n\in\N$,
for a generating collection $\mathring{W}(\EOp\times\POp)\subset W(\EOp\times\POp)$
which is preserved by the degeneracy operators but not by the face operators
of the simplicial structure on $W(\EOp\times\POp)$. In what follows, we say that our operad is quasi-free,
rather than free, to single out such a structure result.
Nevertheless, for simplicity, we still write $W(\EOp\times\POp) = \FOp(\mathring{W}(\EOp\times\POp))$, omitting the simplicial dimension,
and without specifying the category in which we form this relation
as long as this is made clear by the context.
This symmetric sequence $\mathring{W}(\EOp\times\POp)$
consists of decorated trees such that $l_e\not=1$,
for all inner edges $e$. The indecomposable factors of the operadic decomposition of a decorated tree in $W(\EOp\times\POp)$
are the subtrees obtained by cutting all inner edges of length $l_e = 1$.
The components $\mathring{W}(\EOp\times\POp)(r)$ of the symmetric sequence $\mathring{W}(\EOp\times\POp)$
inherit a free action of the symmetric group (we use the free symmetric structure of the cartesian product $\EOp\times\POp$
at this point),
and the quasi-free operad $W(\EOp\times\POp) = \FOp(\mathring{W}(\EOp\times\POp))$
is cofibrant under this condition (see for instance~\cite[Theorem II.8.2.]{FresseBook})

We can now regard the $W$-construction $W(\EOp\times\POp)$ as an operad in bisimplicial sets
with one simplicial dimension inherited from the simplices $\Delta^1$,
which we attach to the inner edges of our decorated trees,
and the other simplicial dimension given by the internal simplicial grading of the operad $\EOp\times\POp$.
We then take the diagonal complex $W(\EOp\times\POp) = \Diag W_{\bullet}((\EOp\times\POp)_{\bullet})$
to retrieve an operad in simplicial sets
from this bisimplicial object.
In the next lemma, we consider the operads in simplicial sets $W((\EOp\times\POp)_n)$, $n\in\N$,
which we form by fixing the internal simplicial dimension of the operad $\EOp\times\POp$
in this bisimplicial object.
We use that the result of Lemma~\ref{lemm:nonunital operad structure} extends to the collections of sets $(\overline{\EOp\times\POp})_n$
which we form by dropping the unit object $(1,1)\in \EOp(1)\times\POp(1)$
from these operads in sets $(\EOp\times\POp)_n = \EOp_n\times\POp_n$, $n\in\N$.
We then get the following statement:

\begin{lemm}\label{lemm:nonunital W}
We have an identity of simplicial operads $W((\EOp\times\POp)_n) = W'((\overline{\EOp\times\POp})_n)$, for each $n\in\N$,
where $W'$ denotes a version of the $W$-construction for non-unital operads
which we define by forgetting about the unit reduction relation of the standard construction (the relation of Eq. \ref{eq:unit contraction}).
\end{lemm}

We now assume $\POp(0) = *$. We note that the operad $\EOp$ satisfies the relation $\EOp(0) = *$ too,
and as a consequence, so does the operad $\EOp\times\POp$.
We accordingly get that the augmentation of the $W$-construction $W(\EOp\times\POp)\rightarrow \EOp\times\POp$
induces a morphism $\tau W(\EOp\times\POp)\rightarrow \EOp\times\POp$
in the category $\Op_*$
by adjunction.
We have an obvious counterpart of the result of the previous lemma for the operad $\tau W((\EOp\times\POp)_n)$.
We use this observation in the verification of the following claim:

\begin{lemm}\label{lemm:W horizontal equivalence}
The augmentation map $\tau W((\EOp\times\POp)_n)(r)\rightarrow(\EOp\times\POp)_n(r)$ defines a weak-equivalence of simplicial sets,
for each dimension $n\in\N$ and for any arity $r\in\N$, where we regard the set $(\EOp\times\POp)_n(r)$
as a discrete simplicial set.
\end{lemm}

\begin{proof}
We set $\NOp_n = (\overline{\EOp\times\POp})_n$ for short and we use the identity $\tau W((\EOp\times\POp)_n) = \tau W'(\NOp_n)$.
We aim to prove that we have a weak-equivalence of simplicial sets $\epsilon: \tau W'(\NOp_n)(r)\xrightarrow{\sim}\NOp_n(r)$,
for each dimension $n\in\NOp$ and for any arity $r\in\N$.
We can take the geometric realization of this map $\epsilon: |\tau W'(\NOp_n)(r)|\rightarrow\NOp_n(r)$
to establish this claim.
We have an obvious map going the other way round $\eta: \NOp_n(r)\rightarrow|\tau W'(\NOp_n)(r)|$,
which merely carries any element $p\in\NOp_n(r)$
to an $r$-corolla with $p$ as label in $\tau W'(\NOp_n)(r)$.
We have $\epsilon\eta = \id$ and we are left to verifying the relation $\eta\epsilon\simeq\id$
in the homotopy category of spaces.

We proceed as follows. Let $\varpi_T = [T;p_x,x\in VT;l_e,e\in\mathring{ET}]$ be a collection
which represents a point of the cell complex $|\tau W'(\NOp_n)(r)|$.
Instead of assigning a length to the internal edges $l_e$, $e\in\mathring{ET}$,
we can equivalently assign a height $h_x\in [0,\infty[$ to each vertex $x\in VT$,
with the following rules (Eq. \ref{eq:height assignment}), in order to parameterize the elements
of our complex:
\begin{equation}\label{eq:height assignment}
h_x = \begin{cases} 0, & \text{if $x$ is the source vertex of the outgoing edge of the tree (the root)}, \\
h_y + l_e, & \text{if $x$ is a source vertex of an inner edge $e\in\mathring{ET}$
with $y$ as target vertex}.
\end{cases}
\end{equation}
For instance, in the case of the decorated $3$-tree
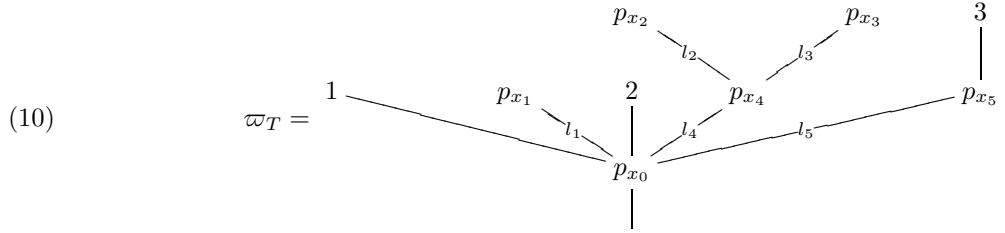
\begin{equation}\label{eq:height assignment example}
\varpi_T = \vcenter{\xymatrix@R=1.5em{ &&& p_{x_2}\ar@{-}[dr]|{l_2} && p_{x_3}\ar@{-}[dl]|{l_3} & 3\ar@{-}[d] \\
1\ar@{-}[drrr] && p_{x_1}\ar@{-}[dr]|{l_1} & 2\ar@{-}[d] & p_{x_4}\ar@{-}[dl]|{l_4} && p_{x_5}\ar@{-}[dlll]|{l_5} \\
&&& p_{x_0}\ar@{-}[d] && \\
&&&&& }},
\end{equation}
with $VT = \{x_0,x_1,x_2,x_3,x_4,x_5\}$, we get $h_{x_0} = 0$, $h_{x_5} = l_5$, $h_{x_4} = l_4$,
$h_{x_3} = l_3 + l_4$, $h_{x_2} = l_2 + l_4$,
and $h_{x_1} = l_1$.
In order to give a sense to this correspondence, we crucially use that no unit reduction relation occurs in $|\tau W'(\NOp_n)(r)|$.
Indeed, the height functions would not be well-defined otherwise. (Thus, we use the result of Lemma~\ref{lemm:nonunital W}
and the structure properties of the cartesian product with the extended Barratt-Eccles operad $\NOp = \EOp\times\POp$
at this point.)

We consider the continuous family $\varpi_T^t\in |\tau W'(\NOp_n)(r)|$, $t\in [0,\infty]$,
defined by making these height parameters
vary by the formula:
\begin{equation}
h^t_x = \min(h_x,t),
\end{equation}
for all $x\in VT$, with the obvious convention $\min(h,\infty) = h$. We readily check that the mapping $(t,\varpi_T)\mapsto\varpi_T^t$
is compatible with the identification relations of the cell complex $|\tau W'(\NOp_n)(r)|$.
We therefore get that this mapping $(t,\varpi_T)\mapsto\varpi_T^t$
defines a continuous family of maps $\rho_t: |\tau W'(\NOp_n)(r)|\rightarrow |\tau W'(\NOp_n)(r)|$
such that $\rho_{\infty} = \id$ and $\rho_0 = \eta\epsilon$,
since making the assignment $h^0_x = \min(h_x,0)$
amounts to assigning the length $l_e = 0$ to all inner edges $e\in\mathring{ET}$ of the tree $T$,
and hence, to contracting these edges in $|\tau W'(\NOp_n)(r)|$.

We can compose the mapping $t\mapsto\rho_t$ with the function $t\mapsto t/(1-t)$
to retrieve a continuous family of maps defined for a value $t\in [0,1]$
of the time parameter $t$, as in the usual definition
of a homotopy. Note that we have $h^t_x\equiv h_x$ for $t\gg h$,
for a bound $h$ that only depends on the tree superstructure $T$,
so that no continuity problem
occurs at $t = \infty$
in our construction.
\end{proof}

We can now complete the:

\begin{proof}[Proof of Theorem~\ref{thm:Goal}]
We use the general statement that a horizontal weak-equivalence of bisimplicial sets $\phi: X_{\bullet n}\xrightarrow{\sim} Y_{\bullet n}$, $n\in\N$,
induces a weak-equivalence when we pass to the diagonal complex $\phi: \Diag X_{\bullet\bullet}\xrightarrow{\sim}\Diag Y_{\bullet\bullet}$
to conclude that the weak-equivalences of the previous lemma $\tau W((\EOp\times\POp)_n)(r)\xrightarrow{\sim}(\EOp\times\POp)_n(r)$, $n\in\N$,
induce a weak-equivalence of simplicial sets $\tau W(\EOp\times\POp)(r)\xrightarrow{\sim}(\EOp\times\POp)(r)$,
for each arity $r\in\N$,
which is nothing but the claim of Theorem~\ref{thm:Goal}.
\end{proof}

Recall that Theorem~\ref{thm:MainResult} is a corollary of Theorem~\ref{thm:Goal}.
Thus, the previous verification also completes the proof of Theorem~\ref{thm:MainResult}.\qed

\subsection{Remark}\label{subsec:homotopy remark}
In the introduction of the paper, we mention that we need to adapt the classical proof that the $W$-construction $W(\POp)$ is equivalent to the given operad $\POp$
in the case of the unitary operad $\tau W(\POp)$.
In short, in the classical construction, authors use the same definition as ours to define a section $\eta: \POp(r)\rightarrow W(\POp)(r)$
of the augmentation map $\epsilon: W(\POp)\rightarrow\POp$ in each arity. Then one makes the length of edges vary by the formula $l_e^t = \min(l_e,t)$
in a decorated tree in order to get a homotopy between the identity map on $W(\POp)$
and the composite $\eta\epsilon: W(\POp)\rightarrow W(\POp)$, and the conclusion that $\epsilon: W(\POp)\rightarrow\POp$
defines a weak-equivalence follows (see for instance~\cite[\S III.1]{BoardmanVogt}).
This construction does not work in the case of the unitary operad $\tau W(\POp)$, because the mapping $l_e\mapsto l_e^t = \min(l_e,t)$
does not preserve the extra reduction relation (Eq. \ref{eq:reduction relation})
which we implement in $\tau W(\POp)$.
Therefore, in the proof of Lemma~\ref{lemm:W horizontal equivalence}, we define a homotopy by making the height of edges vary rather than the length.

\subsection{Remark}\label{subsec:remark}
In \S\ref{subsec:BarrattEccles}, we can identify the category of algebras associated to the operad $\MCat$
with the category of categories $\CCat$
equipped with a strictly associative tensor product $\otimes: \CCat\times\CCat\rightarrow\CCat$
and an object $e\in\CCat$, which is strictly idempotent $e\otimes e = e$,
but which only satisfies the unit relations of tensor products up to natural isomorphisms
in general $x\otimes e\simeq x\simeq e\otimes x$ (compare with the statement of \cite[Theorem I.6.3.2-6.3.3]{FresseBook}).
The tensor product $\otimes: \CCat\times\CCat\rightarrow\CCat$ represents the image of the generating operation $\mu\in\Ob\MCat(2)$
under the functor $\phi: \MCat\rightarrow\End_{\CCat}$ that encodes the action of the operad $\MCat$
on $\CCat$, where $\End_{\CCat}$ denotes the endomorphism operad
of $\CCat$
in the category of categories.
The object $e\in\CCat$ represents the image of the zero-ary operation $e\in\Ob\MCat(0)$.
The natural transformations $x\otimes e\simeq x\simeq e\otimes x$, for $x\in\CCat$,
are given by the image of the corresponding isomorphisms $x_1 e\simeq x_1\simeq e x_1$
in the morphism sets of the category $\MCat(1)$.
(We refer to \cite[\S I.6.3]{FresseBook} for a detailed proof of several variants of this correspondence.)

\section{The case of $k$-truncated operads}

\setcounter{thm}{0}
\renewcommand{\thethm}{\arabic{thm}'}

We now examine the proof of the counterpart of our main statements for $k$-truncated operads.
We mainly briefly check that our constructions and argument lines go through (or can be adapted) in this setting.
We recall the definition of the category of $k$-truncated operads first.
We review the proof of our main results afterwards.

We formally call ``$k$-truncated operad'' the structure formed by an operad $\POp$ whose components $\POp(r)$ are only defined for $r\leq k$,
and where we restrict ourselves to composition products $\circ_i: \POp(m)\times\POp(n)\rightarrow\POp(m+n-1)$
that preserve this arity bound.
This condition is equivalent to the relation $m+n-1\leq k$ in the case $n>0$,
and to the relation $m\leq k$ in the case $n=0$.
In the definition of a $k$-truncated operad, we also restrict the application of the associativity relation of operads
to the cases where the composition products involved in the relation preserve
the arity bound, so that our relation makes sense.

We use the notation $\Op^{\leq k}$ for this category of $k$-truncated operads, and the notation $\Op_*^{\leq k}\subset\Op^{\leq k}$
for the associated subcategory of unitary operads, for which we assume $\POp(0) = *$.
We can also adapt the definition of the notion of a $\Lambda$-operad
to the $k$-truncated setting.
We then consider operads equipped with restriction operators $u^*: \POp(n)\rightarrow\POp(m)$
defined for all injective maps $u: \{1<\dots<m\}\rightarrow\{1<\dots<n\}$
such that $1\leq m\leq n\leq k$.
We also restrict ourselves to the cases where the arity bound is preserved in the expression of equivariance
of composition products $\circ_i$ with respect to these restriction
operators in~\cite[Proposition I.2.2.16]{FresseBook}.
We use the notation $\Lambda\Op_{\varnothing}^{\leq k}$ for this $k$-truncated analogue of the category of $\Lambda$-operads.
We again have an isomorphism of categories $\Op_*^{\leq k}\cong\Lambda\Op_{\varnothing}^{\leq k}$.
We mostly deal with the category $\Op_*^{\leq k}$ (rather than $\Lambda\Op_{\varnothing}^{\leq k}$) in what follows.

The definition of free operads has an obvious counterpart in the $k$-truncated context, which gives a left adjoint $\FOp^{\leq k}: \Seq^{\leq k}\rightarrow\Op^{\leq k}$
of the obvious forgetful functor $\omega: \Op^{\leq k}\rightarrow\Seq^{\leq k}$ from the category of $k$-truncated operads $\Op^{\leq k}$
to the category of $k$-truncated symmetric sequences $\Seq^{\leq k}$ (the category of symmetric sequences $\MOp$
with components $\MOp(r)$ defined for $r\leq k$).
Recall that the ordinary free operad $\FOp(\MOp)$ generated by a symmetric sequence $\MOp$
consists of decorated trees $T$ with $r$-ingoing edges, numbered from $1$ to $r$,
and whose vertices $x$ are labeled by elements $\xi_x\in\MOp(r_x)$
of the symmetric sequence $\MOp$,
where we again use the notation $r_x$ to denote the number of ingoing edges
of our vertex $x$ in the tree $T$.
In the case of $k$-truncated operads $\FOp^{\leq k}(\MOp)$, we just restrict ourselves to the case $r\leq k$,
and we assume $r_x\leq k$ for all vertices of our trees $x\in VT$.
The adjunction morphism $\lambda: \FOp^{\leq k}(\POp)\rightarrow\POp$
carries any such decorated tree with $\MOp = \POp$
to a corresponding treewise composite operation
in the operad $\POp$.
Let us mention that some care is necessary in the context of $k$-truncated operads since some intermediate composites
which we may form by contracting subtrees in a treewise tensor product (see~\cite[\S A.2.5]{FresseBook})
go beyond our arity bound. To avoid this problem, we can evaluate all composites with factors of arity zero at first.
Then we obtain a treewise tensor product shaped on a tree with $r\leq k$ ingoing edges
and in which all indecomposable factors have a positive arity $r_x>0$.
This property ensures that all partial composites which we may form inside our treewise tensor product
do not go above our arity bound.

We can adapt the definition of the projective model structure of operads (in simplicial sets,
in topological spaces) to the category of $k$-truncated operads $\Op^{\leq k}$.
We again assume that a morphism of $k$-truncated operads is a weak-equivalence (respectively, a fibration)
if this morphism forms a weak-equivalence (respectively, a fibration)
in the base category (of simplicial sets, of topological spaces)
arity-wise, and we characterize the cofibrations by the left lifting property with respect to the class
of acyclic fibrations.
We take the morphisms of free objects $\FOp^{\leq k}(i): \FOp^{\leq k}(\MOp)\rightarrow\FOp^{\leq k}(\NOp)$
induced by generating (acyclic) cofibrations
of the category of $k$-truncated symmetric sequences
as a set of generating (acyclic) cofibrations in this model category $\Op^{\leq k}$.
We can also adapt the definition of the Reedy model category of $\Lambda$-operads $\Lambda\Op_{\varnothing}$ in \cite[\S II.8.4]{FresseBook}
to the category of $k$-truncated $\Lambda$-operads $\Lambda\Op_{\varnothing}^{\leq k}$.
We again use the isomorphism of categories $\Op_*^{\leq k}\cong\Lambda\Op_{\varnothing}^{\leq k}$ to transport this model structure
to the category of $k$-truncated unitary operads $\Op_*^{\leq k}$. For our purpose, we can equivalently equip
this category $\Op_*^{\leq k}$ with a restriction of the projective model structure
on $\Op^{\leq k}$.

We can still form a Quillen adjunction between our model categories of $k$-truncated operads $\tau: \Op^{\leq k}\rightleftarrows\Op_*^{\leq k} :\iota$
by taking the canonical category embedding $\iota: \Op_*^{\leq k}\hookrightarrow\Op^{\leq k}$
on the one hand and the obvious $k$-truncated counterpart of our unitarization
functor $\tau: \Op^{\leq k}\rightarrow\Op_*^{\leq k}$ on the other hand.
We then have the following counterpart of the result of Theorem~\ref{thm:MainResult}:

\begin{thm}\label{thm:MainResult:kTruncatedOperads}
The functor $\iota: \Op_*^{\leq k}\hookrightarrow\Op^{\leq k}$ induces a weak-equivalence on derived mapping spaces:
\begin{equation*}
\Map_{\Op_*^{\leq k}}^h(\POp,\QOp)\sim\Map_{\Op^{\leq k}}^h(\iota\POp,\iota\QOp),
\end{equation*}
for all $k$-truncated operads $\POp,\QOp\in\Op_*^{\leq k}$.
\end{thm}

We still reduce the proof of this statement to the verification that we have a weak-equivalence $\tau\ROp\xrightarrow{\sim}\POp$,
for a cofibrant resolution $\ROp\xrightarrow{\sim}\iota\POp$
of the object $\iota\POp$
in the category of all $k$-truncated operads $\Op^{\leq k}$.

We take $\ROp = W^{\leq k}(\EOp\times\POp)$, where $W^{\leq k}(-)$ denotes a $k$-truncated analogue of the $W$-construction,
and $\EOp\times\POp$ denotes the $k$-truncated operad such that $(\EOp\times\POp)(r) = \EOp(r)\times\POp(r)$,
for $r\leq k$, with $\EOp$ defined as in~\S\ref{subsec:BarrattEccles}.
We use the same construction as in~\S\S\ref{subsec:Wconstruction}-\ref{subsec:nonunitalWconstruction}
to define this $k$-truncated version of the $W$-construction (in both the topological setting and the simplicial setting).
We just restrict ourselves to decorated trees such that every subtree $S$ that we may form within a component
delimited by edges of length $l_e = 1$
of our decorated tree
has at most $k$ ingoing edges.
In particular, we assume that the vertices of our trees $x$ have at most $k$ ingoing edges each, so that the corresponding labels $p_x\in\POp(r_x)$
satisfy our arity bound $r_x\leq k$.
Our condition also ensures that the edge contraction relations in the definition of our object (Eq. \ref{eq:edge contraction})
produce allowable composition operations in our $k$-truncated operad $\POp$,
and that $\ROp = W^{\leq k}(\EOp\times\POp)$
forms a quasi-free object
in the category of $k$-truncated operads $W^{\leq k}(\EOp\times\POp) = \FOp^{\leq k}(\mathring{W}{}^{\leq k}(\EOp\times\POp))$
with the same definition as in~\S\ref{subsec:nonunitalWconstruction}
for the generating symmetric sequence $\mathring{W}{}^{\leq k}(\EOp\times\POp)$ (our condition implies
that the treewise tensor products that form this symmetric sequence
have an arity $r\leq k$).
We still deduce from this observation that $W^{\leq k}(\EOp\times\POp)$ forms a cofibrant object in the category of $k$-truncated operads.

We moreover have a canonical morphism $W^{\leq k}(\EOp\times\POp)\rightarrow\EOp\times\POp$,
which we obtain by contracting the edges of our decorated trees
and by performing the corresponding composites
in $\EOp\times\POp$. We can still prove that this morphism is a weak-equivalence, but some care is needed there,
since we have to adapt our construction in order to ensure
that our contracting homotopy produces decorated trees
that fulfill the arity bound conditions of the operad $W^{\leq k}(\EOp\times\POp)$.
We can proceed in two steps. In a first step, we can apply the contracting homotopy of the proof of Lemma~\ref{lemm:W horizontal equivalence}
to the maximal subtrees of the form
\begin{equation*}
\vcenter{\xymatrix@R=1.5em{ *+<2pt>{*}\ar@{-}[dr] & \cdots\ar@{}[d]|{\textstyle S} & *+<2pt>{*}\ar@{-}[dl] \\
& \ar@{-}[d]^{l_0} & \\
& \cdots & }}
\end{equation*}
in a decorated tree, where we now use the notation $*$ for any arity zero element
of our operad $\POp$ (we do not necessarily assume that $\POp$
belongs to the subcategory $\Op_*$ for the moment). For instance, in the case of the tree of Eq. \ref{eq:height assignment example},
we consider the subtrees:
\begin{equation*}
S_1 = \vcenter{\xymatrix@R=1.5em{ p_{x_1}\ar@{-}[d]|(0.6){l_1} & \\ & }}
\!\!\!\!\!\!\!\!\!\!\text{and}\quad S_2 = \vcenter{\xymatrix@R=1.5em{ p_{x_2}\ar@{-}[dr]|{l_2} && p_{x_3}\ar@{-}[dl]|{l_3} \\
& p_{x_4}\ar@{-}[d]|(0.6){l_4} & \\
&& }}.
\end{equation*}
This operation has the effect of reducing the composites with operations of arity zero in our object
and of eliminating the vertices with no ingoing edge.
Thus, as a result of this first contracting homotopy operation,
we get a decorated tree with $r\leq k$ ingoing edges
and of which all vertices $x\in VT$
have $r_x>0$ ingoing edges.
Then we can apply our contracting homotopy a second time, to our whole decorated tree this time, in order to abut
to a corolla, which corresponds to an element of the operad $\EOp\times\POp$
inside $W^{\leq k}(\EOp\times\POp)$.
The condition $r_x>0$ ensures that all intermediate composites which we form in this second retraction process
fulfill our arity bound and define allowable elements
of the $k$-truncated $W$-construction. Thus, we eventually conclude that our morphism $W^{\leq k}(\EOp\times\POp)\rightarrow\EOp\times\POp$
forms a homotopy equivalence arity-wise, and hence, defines a weak-equivalence
in the category of $k$-truncated operads, like the canonical projection $\EOp\times\POp\rightarrow\POp$,
so that $W^{\leq k}(\EOp\times\POp)$ forms a cofibrant resolution of our object $\POp$
in this category $\Op^{\leq k}$.

We now assume that $\POp$ is a unitary $k$-truncated operads in the category of simplicial sets, so that $\POp(0) = *$.
We can readily adapt the observation of Lemma~\ref{lemm:nonunital W} in the $k$-truncated context,
and we can also use the above two-step process
to adapt the definition
of the contracting homotopy of the proof of Lemma~\ref{lemm:W horizontal equivalence}
to the case of the unitarization of the $k$-truncated $W$-construction $\tau W^{\leq k}(\EOp\times\POp)$.
We therefore have the following statement:

\begin{thm}\label{thm:Goal:kTruncatedOperads}
We have $\tau W^{\leq k}(\EOp\times\POp)\xrightarrow{\sim}\EOp\times\POp\xrightarrow{\sim}\POp$,
for any $k$-truncated unitary operad $\POp\in\Op_*^{\leq k}$.\qed
\end{thm}

Recall again that this theorem gives the result of Theorem~\ref{thm:MainResult:kTruncatedOperads} when we take $\ROp = W^{\leq k}(\EOp\times\POp)$
as a cofibrant resolution of the $k$-truncated operad $\POp\in\Op_*^{\leq k}$
to compute our derived mapping spaces.
Thus, the proof of this first statement, Theorem~\ref{thm:MainResult:kTruncatedOperads}, is complete.
\qed

\bibliographystyle{plain}
\bibliography{OperadHomotopySubcategory}

\end{document}